\documentclass[12pt]{amsart}
\usepackage{amsmath,amsthm,latexsym,amscd,amsbsy,amssymb,url}
\usepackage[all]{xypic}
 \relax

\chardef\bslash=`\\ % p. 424, TeXbook

\makeatletter
\def\verbatim{\interlinepenalty\@M \@verbatim
  \leftskip\@totalleftmargin\advance\leftskip2pc
  \frenchspacing\@vobeyspaces \@xverbatim}
\makeatother
\newtheorem{thm}{Theorem}[section]
\newtheorem{cor}[thm]{Corollary}
\newtheorem{lem}[thm]{Lemma}
\newtheorem{pro}[thm]{Proposition}

\theoremstyle{definition}
\newtheorem{defin}{Definition}[section]

\theoremstyle{remark}

\numberwithin{equation}{section}

\begin{document}

%%%%%%% Begin Topmatter %%%%%%%%%%

\title
{$Z$-set unknotting in uncountable products of reals}
\author{A. Chigogidze}
\address{Department of Mathematics and Statistics,
University of North Carolina at Greensboro,
Greensboro, NC, 27402, USA}
\email{chigogidze@uncg.edu}

%%%%%%% End topmatter %%%%%%%%%

\begin{abstract}{We prove a version of $Z$-set unknotting theorem for uncountable products of real numbers. }
\end{abstract}

\maketitle
\markboth{A.~Chigogidze}
{$Z$-set unknotting}

\section{Introduction}\label{S:intro}
Every homeomorphism between $Z$-sets of the countable (infinite) power $R^{\omega}$ of the real line can be extended to an autohomeomorphism of $R^{\omega}$. This result is widely known and often used in infinite-dimensional topology \cite{bp}. Parametric version of this statement is also valid (and used) even though its proof (attributed to H.Toru\'{n}czyk) has never been published. In this paper we introduce, for any infinite cardinal $\tau$, a concept of $Z_{\tau}$-set and prove, using Toru\'{n}czyk's theorem, a fibered version of $Z_{\tau}$-unknotting theorem for uncountable power $R^{\tau}$ of the real line (Theorem \ref{T:ztauunknottingsigmaf}).  As a corollary we obtain a $Z_{\tau}$-unknotting theorem (Corollary \ref{C:1}) for $R^{\tau}$. Similar results (Theorem \ref{T:ztauunknottingsigmafI} and Corollary \ref{C:1I}) are valid for the Tychonov cube $I^{\tau}$ as well. Using our approach we positively answer (Corollary \ref{C:2}) question from \cite{vds} whether any homeomorphism between closed $C$-embedded $\sigma$-compact subsets of $R^{\tau}$ admits an extension to an autohomeomorphism of $R^{\tau}$.

\section{Preliminaries}\label{S:pre}
We refer reader to \cite{chibook} for needed facts about inverse spectra -- in particular, for various versions of the Spectral Theorem. Properties of $C$-embedded subsets also can be found there.

Here we collect certain facts from infinite-dimensional topology which will be needed below. By $\operatorname{cov}(X)$ we denote the collection of all countable functionally open covers of a space $X$.

\begin{defin}\label{D:z-set}
A closed subset $Z$ of a space $X$ is a $Z$-set if for each ${\mathcal U} \in \operatorname{cov}(X)$ there exists a map $f \colon X \to X$ such that $\operatorname{cl}(f(X)) \cap Z = \emptyset$ and $f$ is ${\mathcal U}$-close to $\operatorname{id}_{X}$.
\end{defin}

\begin{defin}\label{D:fz_set}
Let $p \colon X \to Y$ be a map. A closed subset $Z$ of $X$ is a fibered $Z$-set (with respect to $p$) if for each ${\mathcal U} \in \operatorname{cov}(X)$ there exists a map $f \colon X \to X$ such that $\operatorname{cl}(f(X)) \cap Z = \emptyset$, $pf = p$ and $f$ is ${\mathcal U}$-close to $\operatorname{id}_{X}$.
\end{defin}

As noted above, the following result is due to H.~Toru\'{n}czyk. 

\begin{thm}\label{T:fz-setunknottingR}
Let $Z$ and $F$ be fibered $Z$-sets in $R^{\omega}\times R^{\omega}$ with respect to the projection $\pi_{1} \colon R^{\omega} \times R^{\omega} \to R^{\omega}$. Then every homeomorphism $h \colon Z \to F$ such that $\pi_{1}h = \pi_{1}|Z$ extends to a homeomorphism $H \colon R^{\omega}\times R^{\omega} \to R^{\omega}\times R^{\omega}$ such that $\pi_{1}H = \pi_{1}$.
\end{thm}

A counterpart of the above statement for Hilbert cube fibrations appears in \cite{tw}.

\begin{thm}\label{T:fz-setunknottingI}
Let $Z$ and $F$ be fibered $Z$-sets in $I^{\omega}\times I^{\omega}$ with respect to the projection $\pi_{1} \colon I^{\omega} \times I^{\omega} \to I^{\omega}$. Then every homeomorphism $h \colon Z \to F$ such that $\pi_{1}h = \pi_{1}|Z$ extends to a homeomorphism $H \colon I^{\omega}\times I^{\omega} \to I^{\omega}\times I^{\omega}$ such that $\pi_{1}H = \pi_{1}$.
\end{thm}

Recall that a map $f \colon X \to Y$ of Polish spaces is soft if for any Polish space $B$, its closed subset $A \subset B$, and any two maps $g \colon A \to X$, $h \colon B \to Y$ with $fg = h|A$, there exists a map $k \colon B \to X$ such that $g = k|A$ and $fk = h$. In general case (of non-metrizable spaces) we refer reader to \cite[Definition 6.1.12]{chibook}.

The following statement will be used below.

\begin{pro}\label{P:zsetsection}
Let ${\mathcal S} = \{ X_{n}, p^{n+1}_{n},\omega\}$ be an inverse sequence consisting of Polish spaces and soft projections and $Z$ be a closed subset of $\lim{\mathcal S}$. Suppose that for each $n \in \omega$ there exists a section $s_{n}^{n+1} \colon X_{n} \to X_{n+1}$ of the projection $p_{n}^{n+1} \colon X_{n+1} \to X_{n}$ such that $s_{n}^{n+1}(X_{n}) \cap p_{n+1}(Z) = \emptyset$. Then $Z$ is a fibered $Z$-set in $X$ with respect to the projection $p_{0} \colon X \to X_{0}$.
\end{pro}
\begin{proof}
The limit projection $p_{0} \colon X\to X_{0}$ is a soft map and $X$ is a Polish space. For each $n \in \omega$ equip $X_{n}$ with a metric $d_{n}$ bounded by $2^{-n}$. On $X$ consider the metric, defined as follows

\[ d(\{ x_{n}\},\{ x_{n}^{\prime}\}) = \max\{ d_{n}(x_{n},x_{n}^{\prime})\colon n \in \omega\} .\]

Since $R^{\omega}$ is an absolute retract there exists a map $h \colon R^{\omega} \times R^{\omega} \times [0,\infty) \to R^{\omega}$ such that $h(a,b,t) = a$ for each $t \leq 1$ and $h(a,b,t) = b$ for each $t \geq 2$.

For each $n \in \omega$ let $i_{n+1} \colon X_{n+1} \hookrightarrow R^{\omega}$ be a closed embedding and note that the diagonal product $p_{n}^{n+1}\triangle i_{n+1} \colon X_{n+1} \to X_{n} \times R^{\omega}$ is also a closed embedding. Consequently, softness of the projection $p_{n}^{n+1} \colon X_{n+1} \to X_{n}$ guarantees existence of a map $r_{n+1} \colon X_{n}\times R^{\omega} \to X_{n+1}$ such that $p_{n}^{n+1}r_{n+1} = \pi_{X_{n}}$ and $r_{n+1}|(p_{n}^{n+1}\triangle i_{n+1})(X_{n+1}) = (p_{n}^{n+1}\triangle i_{n+1})^{-1}$, where $\pi_{X_{n}} \colon X_{n} \times R^{\omega} \to X_{n}$ is the projection.

Our goal is to construct, for each $\alpha \colon X \to (0,1)$, a map $g_{\alpha} \colon X \to X$ such that 
\begin{itemize}
  \item[(a)]
  $p_{0}g_{\alpha} = p_{0}$;
  \item[(b)] 
  $d(x,g_{\alpha}(x)) \leq \alpha(x)$;
  \item[(c)]
  $\operatorname{cl}(g_{\alpha}(X)) \cap Z = \emptyset$. 
\end{itemize}   
The map $g_{\alpha}$ is constructed as the limit $g_{\alpha} = \lim\{ g_{n} \colon n \in \omega\}$ of maps $g_{n} \colon X \to X_{n}$.

Let $g_{0} = p_{0}$ and suppose that maps for each $i$, with $0 \leq i \leq n$, we have already constructed a map $g_{i} \colon X \to X_{i}$ so that 

\begin{itemize}
  \item[(1)$_{i}$]
   $p^{i}_{i-1}g_{i} = g_{i-1}$;
  \item[(2)$_{i}$] 
  If $\alpha(x) \leq 2^{-i}$, then $g_{i}(x) = p_{i}(x)$
  \item[(3)$_{i}$]
  If $\alpha(x) \geq 2^{-(i-1)}$, then $g_{i}(x) = s_{i-1}^{i}(p_{i-1}(x))$;
\end{itemize}  
 
Next we define the map $g_{n+1}$ by letting

\[ g_{n+1} = r_{n+1}(g_{n} \triangle h(i_{n+1}p_{n+1}\triangle i_{n+1}s_{n}^{n+1}p_{n} \triangle 2^{n+1}\alpha )) . \]

\noindent It is clear that conditions (1)$_{n+1}$--(3)$_{n+1}$ are satisfied.

Obviously, $p_{n}g_{\alpha} = p_{n}$ for each $n \in \omega$. In particular, $p_{0}g_{\alpha} = p_{0}$. If $x \in X$, then there is an index $n \in \omega$ such that $2^{-(n+1)} \leq \alpha(x) \leq 2^{-n}$. Therefore, for each $i \leq n$, we conclude, by condition (2)$_{i}$, that $g_{i}(x) = p_{i}(x)$. This shows that
\[ d(x,g_{\alpha}(x)) = \max\{ d_{k}(p_{k}(x),p_{k}(g_{\alpha}(x))) \colon k \geq n+1\} \leq 2^{-(n+1)}\leq \alpha (x)  \]

\noindent and establishes the required closeness of $\operatorname{id}_{X}$ and $g_{\alpha}$.

Finally, let us show that $\operatorname{cl}(g_{\alpha}(X)) \cap Z = \emptyset$. Consider a sequence $\{ x_{i} \colon i \in \omega\}$ of points in $X$ such that the sequence $\{ g_{\alpha}(x_{i}) \colon i \in \omega\}$ converges to $x \in X$. First observe that $\operatorname{inf}\{ \alpha(x_{i}) \colon i \in \omega \} >0$. Indeed, if this is not the case, then pick a subsequence $\{ x_{i_{j}} \colon j \in \omega\}$ such that $\lim\{ \alpha(x_{i_{j}}) \colon j \in \omega\}=0$. Since, as shown above, $d(y,\alpha(y)) \leq \alpha(y)$ for each $y \in X$, it follows that $\lim \{ x_{i_{j}} \colon j \in \omega\} = x$. This implies, by continuity of $\alpha$, the contradictory equality $\alpha (x) = 0$. Consequently, $0 < \epsilon = \operatorname{inf}\{ \alpha(x_{i}) \colon i \in \omega \}$. Take $n$ large enough so that $2^{-n} \leq \epsilon$. Then $p_{n+1}(g_{\alpha}(x_{i})) = g_{n+1}(x_{i}) = s_{n}^{n+1}(p_{n}(x_{i}))$ for each $i \in \omega$ and 
\begin{multline*}
 p_{n+1}(x) = p_{n+1}(\lim\{ g_{\alpha}(x_{i}) \colon i \in \omega\}) = \lim\{ p_{n+1}(g_{\alpha}(x_{i})) \colon i \in \omega\} =\\ \lim\{ s_{n}^{n+1}(p_{n}(x_{i})) \colon i \in \omega\} \in s_{n}^{n+1}(X_{n}).
\end{multline*}

\noindent Since $s_{n}^{n+1}(X_{n}) \cap p_{n+1}(Z) = \emptyset$ it follows that $x \notin Z$ and therefore $\operatorname{cl}(g_{\alpha}(X)) \cap Z = \emptyset$.
\end{proof}

%%%%%%%%%%%%%%%%%%%%%%%%%%%%%%%%%%%%%%%%%%%%%%%%%%%%%%%%%%%%%%%%%%%
%%%%%%%%%%%%%%%%%%%%%%%%%%%%%%%%%%%%%%%%%%%%%%%
 
\section{$Z_{\tau}$-sets}\label{S:zsets}

We begin by introducing the following notation

\[ B(f,\{ {\mathcal U}_{t} \colon t \in T\} ) = \{ g \in C(X,Y) \colon g\; \text{is}\; {\mathcal U}_{t}\text{-close to}\; f \;\text{for each}\; t \in T\} , \]
\medskip

Let $\tau$ be an infinite cardinal. If $X$ and $Y$ are  Tychonov spaces then $C_{\tau}(X,Y)$ denotes the space of all continuous maps $X \to Y$ with the topology defined as follows (\cite[p.273]{chibook}):  a set $G \subseteq C_{\tau}(X,Y)$ is open if for each $h \in G$ there exists a collection $\{ {\mathcal U}_{t} \colon t \in T\} \subseteq \operatorname{cov}(Y)$, with $|T| < \tau$, such that

\[ h \in B(h,\{ {\mathcal U}_{t} \colon t \in T\} ) \subseteq G .\]

Obviously if $\tau = \omega$, then the above topology coincides with the limitation topology (see \cite{bp}).

In a wide range of situations description of basic neighborhoods in $C_{\tau}(Y,X)$ is quite simple. Proof of the following statement follows \cite[Lemma 6.5.1]{chibook} and is therefore omitted.

\begin{lem}\label{L:description}
Let $\tau > \omega$ and $X$ be a $z$-embedded subspace of a product $\prod\{ X_{t} \colon t \in T\}$ of separable metrizable spaces. If $|T| = \tau$, then basic neighborhoods of a map $f \colon Y \to X$ in $C_{\tau}(Y,X)$ are of the form $B(f, S) = \{ g \in C_{\tau}(Y,X) \colon \pi_{S}g = \pi_{S}f\}$, $S \subseteq T$, $|S| < \tau$, where $\pi_{S} \colon \prod\{ X_{t} \colon t \in T\} \to \prod\{ X_{t} \colon t \in S\}$ denotes the projection onto the corresponding subproduct.
\end{lem}

Now we are ready to define (fibered) $Z_{\tau}$-sets.

\begin{defin}\label{D:zset}
Let $\tau \geq \omega$. A closed subset $Z \subseteq X$ is a $Z_{\tau}$-set in $X$ if the set $\{ f \in C_{\tau}(X,X) \colon f(X) \;\text{and}\; Z \;\text{are functionally separated}\}$ is dense in the space $C_{\tau}(X,X)$. 
\end{defin}

For a map $p \colon X \to Y$ by $C_{\tau}^{p}(X,X)$ we denote set of those $f \in C_{\tau}(X,X)$ for which $pf = p$.
\begin{defin}\label{D:fzset}
Let $p \colon X \to Y$ be a map. A closed subset $Z$ of $X$ is a fibered $Z_{\tau}$-set (with respect to $p$) in $X$ if the set $\{ f \in C_{\tau}^{p}(X,X) \colon f(X) \;\text{and}\;Z\;\text{are functionally separated}\}$ is dense in the space $C_{\tau}^{p}(X,X)$.
\end{defin}

Obviously in case of compact $X$ in both definitions it would suffice to require only that $Z \cap f(X) = \emptyset$.

%%%%%%%%%%%%%%%%%%%%%%%%%%%%%%%%%%%%%%%%%%%%%%%%%%%%%%%%%%%%%%%%%%%%%%%%%%%%%%%%%%%%%%%%%%%%%%%
%%%%%%%%%%%%%%%%%%%%%%%%%%%%%%%%%%%%%%%%%%%%%%%%%%%%%%%%%%%%%%%%%%%%%%%%%%%%%%%%%%%%%%%%%%%%%%%
%%%%%%%%%%%%%%%%%%%%%%%%%%%%%%%%%%%%%%%%%%%%%%%%%%%%%%%%%%%%%%%%%%%%%%%%%%%%%%%%%%%%%%%%%%%%%%%
%%%%%%%%%%%%%%%%%%%%%%%%%%%%%%%%%%%%%%%%%%%%%%%%%%%%%%%%%%%%%%%%%%%%%%%%%%%%%%%%%%%%%%%%%%%%%%%
%%%%%%%%%%%%%%%%%%%%%%%%%%%%%%%%%%%%%%%%%%%%%%%%%%%%%%%%%%%%%%%%%%%%%%%%%%%%%%%%%%%%%%%%%%%%%%%
%%%%%%%%%%%%%%%%%%%%%%%%%%%%%%%%%%%%%%%%%%%%%%%%%%%%%%%%%%%%%%%%%%%%%%%%%%%%%%%%%%%%%%%%%%%%%%%

\section{$Z_{\tau}$-set unknotting in $R^{\tau}$}\label{S:ztauunknotting}

In this section we prove our main results (Theorem \ref{T:ztauunknottingsigmaf} and Corollary \ref{C:1}). First we need several lemmas. We say that a map $f \colon R^{A} \to R^{A}$ (or a map $f \colon R^{A} \times R^{A} \to R^{A} \times R^{A}$) factors through $C$ by $f_{C}$ if there is a map $f_{C} \colon R^{C} \to R^{C}$ (or $f_{C} \colon R^{C} \times R^{C} \to R^{C} \times R^{C}$) such that $\pi_{C}^{A}f = f_{C}\pi^{A}_{C}$ (or $(\pi_{C}^{A}\times \pi_{C}^{A})f = f_{C}(\pi_{C}^{A}\times \pi_{C}^{A}$), where $\pi^{A}_{C} \colon R^{A} \to R^{C}$ denotes the projection onto the corresponding subproduct. We also say that $f$ is $B$-invariant, where $B$ is a subset of $A$, if $\pi_{B}^{A}f = \pi_{B}^{A}$ (or $(\pi_{B}^{A} \times \pi_{B}^{A})f = \pi_{B}^{A} \times \pi_{B}^{A}$).

\begin{lem}\label{L:1}
Let $|A| > \omega$ and $f \colon R^{A}\times R^{A} \to R^{A}\times R^{A}$ be a map. Then the set ${\mathcal M}_{f}$, consisting of those $C \in \exp_{\omega}A$ for which $f$ factors through $C$ is cofinal and $\omega$-closed in $\exp_{\omega}A$. If, in addition, $f$ is $B$-invariant for some $B \subset A$, then $f$ factors through $B \cup C$ whenever $C \in {\mathcal M}_{f}$. 
\end{lem}
\begin{proof}
By Spectral Theorem, the set ${\mathcal M}_{f} \subseteq \exp_{\omega}A$, consisting of those $C \in \exp_{\omega}A$ for which $f$ factors through $C$ is cofinal and $\omega$-closed in $\exp_{\omega}A$. Straightforward calculation establishes the second part of the lemma.
\end{proof}

\begin{lem}\label{L:2}
Let $|A| > \omega$, $B  \subseteq A$ and $Z$ and $F$ be closed subsets in $R^{A}\times R^{A}$. Suppose that $f, g \colon R^{A}\times R^{A} \to R^{A}\times R^{A}$ are $B$-invariant maps such that $f(Z) = F$, $g(F) = Z$, $gf|Z = \operatorname{id}_{Z}$ and $fg|F = \operatorname{id}_{F}$. Then the set ${\mathcal M}_{f,g} = {\mathcal M}_{f}\cap {\mathcal M}_{g}$ (see Lemma \ref{L:1}) is cofinal and $\omega$-closed in $\exp_{\omega}A$. Moreover, if $Z_{B\cup C} = \operatorname{cl}_{R^{B\cup C}\times R^{B\cup C}}\left(\pi_{B\cup C}^{A}\times \pi_{B\cup C}^{A}\right)(Z)$ and $F_{B\cup C} = \operatorname{cl}_{R^{B\cup C}\times R^{B\cup C}}\left(\pi_{B\cup C}^{A}\times \pi_{B\cup C}^{A}\right)(F)$, then for each $C \in {\mathcal M}_{f,g}$ we have: 
\begin{enumerate}
  \item 
  $f$ factors through $B \cup C$ by $f_{B\cup C}$,
  \item 
   $g$ factors through $B \cup C$ by $g_{B\cup C}$,

\item  
$f_{B \cup C}(Z_{B\cup C}) = F_{B\cup C}$, 
  \item 
$g_{B \cup C}(F_{B\cup C}) = Z_{B\cup C}$,
  \item 
  $g_{B\cup C}f_{B \cup C}|Z_{B\cup C} = \operatorname{id}_{Z_{B\cup C}}$,
  \item 
  $f_{B\cup C}g_{B \cup C}|F_{B\cup C} = \operatorname{id}_{F_{B\cup C}}$.

\end{enumerate}
\end{lem}
\begin{proof}
It suffices to note that, by Lemma \ref{L:1} and \cite[Proposition 1.1.27]{chibook}, the set ${\mathcal M}_{f,g} = {\mathcal M}_{f}\cap {\mathcal M}_{g}$  is still cofinal and $\omega$-closed in $\exp_{\omega}A$. Verification of the listed properties is straightforward.
\end{proof}

\begin{lem}\label{L:3}
Let $|B| \geq \omega$, $C = \cup\{ C_{n} \colon n \in \omega\}$, $|C_{n}| \leq \omega$, $C_{0} = \emptyset$, $Z$ and $F$ be closed sets in $R^{B\cup C}$ and $f, g \colon R^{B\cup C} \to R^{B\cup C}$ be $B$-invariant maps such that $f(Z) = F$, $g(F) = Z$, $gf|Z = \operatorname{id}_{Z}$ and $fg|F = \operatorname{id}_{F}$. Suppose also that for each $n\in \omega$ there are maps $s_{n}^{n+1}, r_{n}^{n+1} \colon R^{B\cup C_{n}} \to R^{B\cup C_{n+1}}$ such that 
\begin{itemize}
  \item[(i)] 
$\pi_{B\cup C_{n}}^{B\cup C_{n+1}}s_{n}^{n+1} = \operatorname{id}_{R^{B\cup C_{n}}}$, 
\item[(ii)]
$\pi_{B\cup C_{n}}^{B\cup C_{n+1}}r_{n}^{n+1} = \operatorname{id}_{R^{B\cup C_{n}}}$, 
\item[(iii)]
$\operatorname{Im}(s_{n}^{n+1})$ and $\pi_{B\cup C_{n+1}}^{B \cup C}(Z)$ are functionally separated in $R^{B\cup C_{n+1}}$, 
\item[(iv)] 
$\operatorname{Im}(r_{n}^{n+1})$ and $\pi_{B\cup C_{n+1}}^{B \cup C}(F)$ are functionally separated in $R^{B\cup C_{n+1}}$ 
\end{itemize}
Then there exists a homeomorphism $H \colon R^{B\cup C} \to R^{B\cup C}$ such that $\pi_{B}^{B\cup C}H = \pi_{B}^{B\cup C}$, and $H|Z= f|Z$.
\end{lem}
\begin{proof}
Let $M \subseteq N$, $|N \setminus M| \leq \omega$ and ${\mathcal M}$ be a cofinal and $\omega$-closed subset of $\exp_{\omega}M$. Then the correspondence $P \mapsto P\cup (N \setminus M)$, $P \in {\mathcal M}$, identifies ${\mathcal M}$ with a cofinal and $\omega$-closed subset of $\exp_{\omega}N$. Using just described correspondence, Spectral Theorem for factorizing $\omega$-spectra \cite[Theorem 1.3.6]{chibook}, applied to the maps $f$, $g$, $s_{n}^{n+1}$, $r_{n}^{n+1}$, and \cite[Proposition 1.1.27]{chibook}, we can find a countable subset $B_{0} \subseteq B$ and maps $\widetilde{f}, \widetilde{g} \colon R^{B_{0}\cup (C\setminus B)} \to R^{B_{0}\cup (C\setminus B)}$, $\widetilde{s}_{n}^{n+1}, \widetilde{r}_{n}^{n+1} \colon R^{B_{0}\cup (C_{n}\setminus B)}\to R^{B_{0}\cup (C_{n+1}\setminus B)}$ such that

\begin{itemize}
  \item[(1)]
$\pi_{B_{0}}^{B_{0}\cup (C\setminus B)}\widetilde{f}= \pi_{B_{0}}^{B_{0}\cup (C \setminus B)}$;
  \item[(2)] 
$\pi_{B_{0}}^{B_{0}\cup (C\setminus B)}\widetilde{g}= \pi_{B_{0}}^{B_{0}\cup (C\setminus B)}$;
\item[(3)]
$\pi_{B_{0}\cup (C\setminus B)}^{B\cup C}f = \widetilde{f}\pi_{B_{0}\cup (C\setminus B)}^{B\cup C}$;
  \item[(4)]
$\pi_{B_{0}\cup (C\setminus B)}^{B\cup C}g = \widetilde{g}\pi_{B_{0}\cup (C\setminus B)}^{B\cup C}$;
  \item[(5)] 
$\widetilde{f}(\operatorname{cl}(\pi_{B_{0}\cup (C\setminus B)}^{B\cup C}(Z))) = \operatorname{cl}(\pi_{B_{0}\cup (C\setminus B)}^{B\cup C}(F))$;
  \item[(6)] 
$\widetilde{g}(\operatorname{cl}(\pi_{B_{0}\cup (C\setminus B)}^{B\cup C}(F))) = \operatorname{cl}(\pi_{B_{0}\cup (C\setminus B)}^{B\cup C}(Z))$;
  \item[(7)] 
 $\pi_{B_{0}\cup (C_{n+1}\setminus B)}^{B\cup C_{n+1}}s_{n}^{n+1} = \widetilde{s}_{n}^{n+1}\pi_{B_{0}\cup (C_{n}\setminus B)}^{B\cup C_{n}}$
  \item[(8)] 
   $\pi_{B_{0}\cup (C_{n+1}\setminus B)}^{B\cup C_{n+1}}r_{n}^{n+1} = \widetilde{r}_{n}^{n+1}\pi_{B_{0}\cup (C_{n}\setminus B)}^{B\cup C_{n}}$
  \item[(9)] 
 $\pi_{B_{0}\cup (C_{n}\setminus B)}^{B_{0}\cup (C_{n+1}\setminus B)}\widetilde{s}_{n}^{n+1} = \operatorname{id}_{R^{B_{0}\cup (C_{n}\setminus B)}}$
  \item[(10)] 
 $\pi_{B_{0}\cup (C_{n}\setminus B)}^{B_{0}\cup (C_{n+1}\setminus B)}\widetilde{r}_{n}^{n+1} = \operatorname{id}_{R^{B_{0}\cup (C_{n}\setminus B)}}$
 \end{itemize}

By increasing $B_{0}$ if necessary and using \cite[Corollary 1.1.28]{chibook}, we may assume in addition that
\begin{itemize}
  \item[(11)] 
$\operatorname{cl}_{R^{B_{0}\cup (C_{n+1}\setminus B)}}(\operatorname{Im}(\widetilde{s}_{n}^{n+1})) \cap \operatorname{cl}_{R^{B_{0}\cup (C_{n+1}\setminus B)}}(\pi^{B \cup C_{n+1}}_{B_{0}\cup (C_{n+1}\setminus B)}(Z)) = \emptyset$;
  \item[(12)] 
$\operatorname{cl}_{R^{B_{0}\cup (C_{n+1}\setminus B)}}(\operatorname{Im}(\widetilde{r}_{n}^{n+1})) \cap \operatorname{cl}_{R^{B_{0}\cup (C_{n+1}\setminus B)}}(\pi^{B \cup C_{n+1}}_{B_{0}\cup (C_{n+1}\setminus B)}(F)) = \emptyset$;
\end{itemize}

By Proposition \ref{P:zsetsection}, $\operatorname{cl}(\pi_{B_{0}\cup (C\setminus B)}^{B\cup C}(Z))$ and $\operatorname{cl}(\pi_{B_{0}\cup (C\setminus B)}^{B\cup C}(F))$ are fibered $Z$-set in $R^{B_{0}\cup (C\setminus B)}$ with respect to the projection $\pi_{B_{0}}^{B_{0}\cup (C\setminus B)} \colon R^{B_{0}\cup (C\setminus B)} \to R^{B_{0}}$. By Theorem \ref{T:fz-setunknottingR}, the homeomorphism 
\[ \widetilde{f}|\operatorname{cl}_{R^{B_{0}\cup (C\setminus B)}}(\pi_{B_{0}\cup (C\setminus B)}^{B\cup C}(Z)) \colon \operatorname{cl}_{R^{B_{0}\cup (C\setminus B)}}(\pi_{B_{0}\cup (C\setminus B)}^{B\cup C}(Z)) \to \operatorname{cl}_{R^{B_{0}\cup (C\setminus B)}}(\pi_{B_{0}\cup (C\setminus B)}^{B\cup C}(F))\]

\noindent can be extended to an autohomeomorphism $\widetilde{H} \colon R^{B_{0}\cup (C\setminus B)} \to R^{B_{0}\cup (C\setminus B)}$ such that $\pi_{B_{0}}^{B_{0}\cup (C\setminus B)}\widetilde{H} = \pi_{B_{0}}^{B_{0}\cup (C\setminus B)}$. Since the following diagram is a pullback square (note that $(B \cup C)\setminus B = (B_{0}\cup (C\setminus B))\setminus B_{0}$)

\[
        \xymatrix{
            R^{B \cup C} \ar^{\pi_{B}^{B\cup C}}[rr] \ar_{\pi_{B_{0}\cup (C\setminus B)}^{B\cup C}}[d] & & R^{B} \ar^{\pi_{B_{0}}^{B}}[d]\\
            R^{B_{0}\cup (C\setminus B)} \ar^{\pi_{B_{0}}^{B_{0}\cup (C\setminus B)}}[rr]   & & R^{B_{0}}   \\
        }
      \]

\noindent it follows that there is a unique homeomorphism $H \colon R^{B \cup C} \to R^{B\cup C}$ such that $\pi_{B_{0}\cup (C\setminus B)}^{B\cup C}H = \widetilde{H}\pi_{B_{0}\cup (C\setminus B)}^{B\cup C}$ and $\pi_{B}^{B\cup C}H= \pi_{B}^{B\cup C}$. Straightforward verification shows that $H|Z = f|Z$.
\end{proof}

\begin{lem}\label{L:limitfiberwise}
Let $\{ C_{n} \colon n \in \omega\}$ be an increasing sequence of countable subsets, $C_{0} = \emptyset$ and $C = \cup \{ C_{n} \colon n \in \omega\}$.  Suppose also that $B$ is an arbitrary set, $Z$ and $F$ are closed subsets of $R^{B\cup C}\times R^{B\cup C}$ and we are given maps $\varphi_{n}, \psi_{n} \colon R^{B\cup C_{n}}\times R^{B\cup C_{n}} \to R^{B\cup C_{n}}\times R^{B\cup C_{n}}$, $ n \geq 1$, such that
\begin{itemize}
  \item[(1)]
$\pi_{1}^{B\cup C_{n}}\varphi_{n} =  \pi_{1}^{B\cup C_{n}}$, $n \geq 1$, where $\pi_{1}^{B\cup C_{n}} \colon R^{B \cup C_{n}}\times R^{B\cup C_{n}} \to R^{B\cup C_{n}}$ denotes the projection onto the first factor;
  \item[(2)] 
$\pi_{1}^{B\cup C_{n}}\psi_{n} =  \pi_{1}^{B\cup C_{n}}$, $n \geq 1$, where $\pi_{1}^{B\cup C_{n}} \colon R^{B \cup C_{n}}\times R^{B\cup C_{n}} \to R^{B\cup C_{n}}$ denotes the projection onto the first factor;
  \item[(3)] 
$\varphi_{n+1}$ is $(B\cup C_{n})$-invariant, $n \in \omega$;
  \item[(4)]
$\psi_{n+1}$ is $(B\cup C_{n})$-invariant, $n \in \omega$;
  \item[(5)] 
$\left(\pi_{B\cup C_{n}}^{B\cup C}\times \pi_{B\cup C_{n}}^{B\cup C}\right)(Z)$ and $\operatorname{Im}(\varphi_{n})$ are functionally separated in $R^{B \cup C_{n}}\times R^{B\cup C_{n}}$, $n \geq 1$;
  \item[(6)] 
$\left(\pi_{B\cup C_{n}}^{B\cup C}\times \pi_{B\cup C_{n}}^{B\cup C}\right)(F)$ and $\operatorname{Im}(\psi_{n})$ are functionally separated in $R^{B \cup C_{n}}\times R^{B\cup C_{n}}$, $n \geq 1$;
\end{itemize}  
If $f, g \colon R^{B \cup C}\times R^{B\cup C} \to R^{B\cup C}\times R^{B\cup C}$ are maps such that 
\begin{itemize}
  \item[(7)] 
 $f(Z) = F$, $g(F) = Z$;
  \item[(8)] 
$gf|Z = \operatorname{id}_{Z}$ and $fg|F = \operatorname{id}_{F}$;
  \item[(9)] 
$\pi_{1}^{B\cup C}f = \pi_{1}^{B\cup C}$, $\pi_{1}^{B\cup C}g = \pi_{1}^{B\cup C}$
  \item[(10)] 
 $\pi_{B\cup C_{0}}^{B\cup C}f = \pi_{B\cup C_{0}}^{B\cup C}|Z$,  $\pi_{B\cup C_{0}}^{B\cup C}g = \pi_{B\cup C_{0}}^{B\cup C}|F$
\end{itemize}
\noindent then the homeomorphism $h = f|Z \colon Z \to F$ can be extended to an autohomeomorphism $H \colon R^{B\cup C}\times R^{B\cup C} \to R^{B\cup C}\times R^{B\cup C}$ such that $\pi_{1}^{B\cup C}H = \pi_{1}^{B\cup C}$ and $\left(\pi_{B}^{B\cup C} \times \pi_{B}^{B\cup C}\right) H = \pi_{B}^{B\cup C} \times \pi_{B}^{B\cup C}$. 
\end{lem}
\begin{proof}
For each $n \geq 1$ let

\[ \varphi_{n}^{\prime} = \varphi_{n}\left(\pi_{B\cup C_{n}}^{B\cup C} \times \operatorname{id}_{R^{B\cup C_{n}}}\right) \triangle \pi_{1} \colon R^{B\cup C} \times R^{B \cup C_{n}} \to R^{B\cup C} \times R^{B \cup C_{n}}\]

\noindent and

\[ \psi_{n}^{\prime} = \psi_{n}\left(\pi_{B\cup C_{n}}^{B\cup C} \times \operatorname{id}_{R^{B\cup C_{n}}}\right) \triangle \pi_{1} \colon R^{B\cup C} \times R^{B \cup C_{n}} \to R^{B\cup C} \times R^{B \cup C_{n}}\]

Here is the corresponding diagram:

\[
        \xymatrix{
          R^{B \cup C}\times R^{B\cup C_{n+1}} \ar@(ul,ur)[]|{\varphi_{n+1}^{\prime}} \ar^{\pi_{1}}[rrrr] \ar_{\pi_{B\cup C_{n+1}}^{B\cup C}\times \operatorname{id}}[dddd] \ar^(.6){\operatorname{id}\times \pi_{B\cup C_{n}}^{B\cup C_{n+1}}}[dr] & & & & R^{B\cup C} \ar^{\pi_{B \cup C_{n+1}}^{B\cup C}}[dddd]\\
  & R^{B\cup C}\times R^{B\cup C_{n}} \ar@(ul,ur)[]|{\varphi_{n}^{\prime}} \ar^{\pi_{1}}[rrru] \ar^(.3){\pi_{B\cup C_{n}}^{B\cup C}\times \operatorname{id}}[ddddddl] & & & \\
 & & & & \\
  & & & & \\
 R^{B\cup C_{n+1}}\times R^{B\cup C_{n+1}} \ar@(ul,ur)[]|{\varphi_{n+1}} \ar_{\pi_{B\cup C_{n}}^{B\cup C_{n+1}}\times \pi_{B\cup C_{n}}^{B\cup C_{n+1}}}[ddd] \ar^{\pi_{1}}[rrrr] & & & & R^{B\cup C_{n+1}} \ar^{\pi_{B\cup C_{n}}^{B\cup C_{n+1}}}[ddd]\\
  & & & & \\
 & & & & \\
             R^{B\cup C_{n}} \times R^{B\cup C_{n}} \ar@(ul,ur)[]|{\varphi_{n}} \ar^{\pi_{1}}[rrrr]   & & & & R^{B \cup C_{n}}   \\
        }
      \]

 It follows from properties (1)--(4) that

\begin{itemize}
\item[(11)]
$\left(\operatorname{id}_{R^{B\cup C}}\times \pi_{B\cup C_{n}}^{B\cup C_{n+1}}\right)\varphi_{n+1}^{\prime} = \operatorname{id}_{R^{B\cup C}}\times \pi_{B\cup C_{n}}^{B\cup C_{n+1}}$
  \item[(12)] 
$\left(\operatorname{id}_{R^{B\cup C}}\times \pi_{B\cup C_{n}}^{B\cup C_{n+1}}\right)\psi_{n+1}^{\prime} = \operatorname{id}_{R^{B\cup C}}\times \pi_{B\cup C_{n}}^{B\cup C_{n+1}}$
\end{itemize}

Properties (5) and (6) guarantee that 

\begin{itemize}
  \item[(13)] 
$\left(\operatorname{id}_{R^{B\cup C}}\times \pi_{B\cup C_{n}}^{B\cup C}\right)(Z)$ and $\operatorname{Im}(\varphi_{n}^{\prime})$ are functionally separated in $R^{B \cup C}\times R^{B\cup C_{n}}$, $n \geq 1$;
  \item[(14)] 
$\left(\operatorname{id}_{R^{B\cup C}}\times \pi_{B\cup C_{n}}^{B\cup C}\right)(F)$ and $\operatorname{Im}(\psi_{n}^{\prime})$ are functionally separated in $R^{B \cup C}\times R^{B\cup C_{n}}$, $n \geq 1$;
\end{itemize}

Properties (9) and (10) imply that
\begin{itemize}
  \item[(15)]
$\left(\operatorname{id}_{R^{B\cup C}} \times \pi_{B\cup C_{0}}^{B\cup C}\right) f = \operatorname{id}_{R^{B\cup C}} \times \pi_{B\cup C_{0}}^{B\cup C}$,  
\item[(16)]
$\left(\operatorname{id}_{R^{B\cup C}} \times \pi_{B\cup C_{0}}^{B\cup C}\right) g = \operatorname{id}_{R^{B\cup C}} \times \pi_{B\cup C_{0}}^{B\cup C}$.
\end{itemize}

By Lemma \ref{L:3}, applied to the collection $\{ (B\cup C_{n}) \times (B\cup C_{n}) \colon n \in \omega\}$, there exists an autohomeomorphism $H$ of $R^{B\cup C} \times R^{B\cup C}$ such that $H|Z = f|Z$ and $\left(\operatorname{id}_{R^{B\cup C}} \times \pi_{B\cup C_{0}}^{B\cup C}\right) H = \operatorname{id}_{R^{B\cup C}} \times \pi_{B\cup C_{0}}^{B\cup C}$. The latter implies that $\pi_{1}^{B\cup C}H = \pi_{1}^{B\cup C}$ and $\left(\pi_{B}^{B\cup C} \times \pi_{B}^{B\cup C}\right) H = \pi_{B}^{B\cup C} \times \pi_{B}^{B\cup C}$ as required.
\end{proof}

\begin{lem}\label{L:limitfiberwise1}
Let $\{ C_{n} \colon n \in \omega\}$ be an increasing sequence of countable subsets, $C_{0} = \emptyset$ and $C = \cup \{ C_{n} \colon n \in \omega\}$.  Suppose also that $B$ is an arbitrary set, $Z$ and $F$ are closed subsets of $R^{B\cup C}\times R^{B\cup C}$ and we are given maps $\varphi_{n}, \psi_{n} \colon R^{B\cup C_{n}}\times R^{B\cup C_{n}} \to R^{B\cup C_{n}}\times R^{B\cup C_{n}}$, $ n \geq 1$, such that
\begin{itemize}
  \item[(1)]
$\pi_{1}^{B\cup C_{n}}\varphi_{n} =  \pi_{1}^{B\cup C_{n}}$, $n \geq 1$, where $\pi_{1}^{B\cup C_{n}} \colon R^{B \cup C_{n}}\times R^{B\cup C_{n}} \to R^{B\cup C_{n}}$ denotes the projection onto the first factor;
  \item[(2)] 
$\pi_{1}^{B\cup C_{n}}\psi_{n} =  \pi_{1}^{B\cup C_{n}}$, $n \geq 1$, where $\pi_{1}^{B\cup C_{n}} \colon R^{B \cup C_{n}}\times R^{B\cup C_{n}} \to R^{B\cup C_{n}}$ denotes the projection onto the first factor;
  \item[(3)] 
$\left(\pi_{B\cup C_{n}}^{B\cup C_{n+1}} \times \pi_{B\cup C_{n}}^{B\cup C_{n+1}}\right) \varphi_{n+1} = \pi_{B\cup C_{n}}^{B\cup C_{n+1}} \times \pi_{B\cup C_{n}}^{B\cup C_{n+1}}$, $n \in \omega$;
  \item[(4)]
$\left(\pi_{B\cup C_{n}}^{B\cup C_{n+1}} \times \pi_{B\cup C_{n}}^{B\cup C_{n+1}}\right) \psi_{n+1} = \pi_{B\cup C_{n}}^{B\cup C_{n+1}} \times \pi_{B\cup C_{n}}^{B\cup C_{n+1}}$, $n \in \omega$;
  \item[(5)] 
$\left(\pi_{B\cup C_{n}}^{B\cup C}\times \pi_{B\cup C_{n}}^{B\cup C}\right)(Z)$ and $\operatorname{Im}(\varphi_{n})$ are functionally separated in $R^{B \cup C_{n}}\times R^{B\cup C_{n}}$, $n \geq 1$;
  \item[(6)] 
$\left(\pi_{B\cup C_{n}}^{B\cup C}\times \pi_{B\cup C_{n}}^{B\cup C}\right)(F)$ and $\operatorname{Im}(\psi_{n})$ are functionally separated in $R^{B \cup C_{n}}\times R^{B\cup C_{n}}$, $n \geq 1$;
\end{itemize}  
Let also $H_{0} \colon R^{B\cup C_{0}}\times R^{B\cup C_{0}} \to R^{B\cup C_{0}}\times R^{B\cup C_{0}}$ be a homeomorphism such that $\pi_{1}^{B\cup C_{0}}H = \pi_{1}^{B\cup C_{0}}$ and $f, g \colon R^{B \cup C}\times R^{B\cup C} \to R^{B\cup C}\times R^{B\cup C}$ be maps such that 
\begin{itemize}
  \item[(7)] 
 $f(Z) = F$, $g(F) = Z$;
  \item[(8)] 
$gf|Z = \operatorname{id}_{Z}$ and $fg|F = \operatorname{id}_{F}$;
  \item[(9)] 
$\pi_{1}^{B\cup C}f = \pi_{1}^{B\cup C}$, $\pi_{1}^{B\cup C}g = \pi_{1}^{B\cup C}$
  \item[(10)] 
 $\left(\pi_{B\cup C_{0}}^{B\cup C}\times \pi_{B\cup C_{0}}^{B\cup C}\right)f = H_{0}\left(\pi_{B\cup C_{0}}^{B\cup C}\times \pi_{B\cup C_{0}}^{B\cup C}\right)$,  $\left(\pi_{B\cup C_{0}}^{B\cup C}\times \pi_{B\cup C_{0}}^{B\cup C}\right)g = H_{0}^{-1}\left(\pi_{B\cup C_{0}}^{B\cup C}\times \pi_{B\cup C_{0}}^{B\cup C}\right)$
\end{itemize}
\noindent then the homeomorphism $h = f|Z \colon Z \to F$ can be extended to an autohomeomorphism $H \colon R^{B\cup C}\times R^{B\cup C} \to R^{B\cup C}\times R^{B\cup C}$ such that $\pi_{1}^{B\cup C}H = \pi_{1}^{B\cup C}$ and $\left(\pi_{B}^{B\cup C} \times \pi_{B}^{B\cup C}\right) H = H_{0}\left(\pi_{B}^{B\cup C} \times \pi_{B}^{B\cup C}\right)$. 
\end{lem}
\begin{proof}
Let $G_{0} = H_{0}$, $G_{n+1} \colon R^{B\cup C_{n+1}}\times R^{B\cup C_{n+1}} \to R^{B\cup C_{n+1}}\times R^{B\cup C_{n+1}}$ be a homeomorphism such that $\left(\pi_{B\cup C_{n}}^{B\cup C_{n+1}}\times  \pi_{B\cup C_{n}}^{B\cup C_{n+1}}\right)G_{n+1} = G_{n}\left(\pi_{B\cup C_{n}}^{B\cup C_{n+1}}\times \pi_{B\cup C_{n}}^{B\cup C_{n+1}}\right)$ and $\pi_{1}^{B\cup C_{n+1}}G_{n+1} = \pi_{1}^{B\cup C_{n+1}}$, $n \in \omega$, and $G = \lim\{ G_{n} \colon n \in \omega\}$. Note that $\pi_{1}^{B\cup C}G = \pi_{1}^{B\cup C}$ and $\left(\pi_{B\cup C_{n}}^{B\cup C}\times \pi_{B\cup C_{n}}^{B\cup C}\right)G = G_{n}\left(\pi_{B\cup C_{n}}^{B\cup C}\times \pi_{B\cup C_{n}}^{B\cup C}\right)$ for each $n \in \omega$. In particular, $\left(\pi_{B\cup C_{0}}^{B\cup C}\times \pi_{B\cup C_{0}}^{B\cup C}\right)G = H_{0}\left(\pi_{B\cup C_{0}}^{B\cup C}\times \pi_{B\cup C_{0}}^{B\cup C}\right)$.

Consider the set $Z^{\prime} =G(Z)$ and maps $f^{\prime} = fG^{-1}$, $g^{\prime} = Gg \colon R^{B\cup C}\times R^{B\cup C} \to R^{B\cup C}\times R^{B\cup C}$. 

\[ f^{\prime}(Z^{\prime}) = f(G^{-1}(G(Z))) = f(Z) = F\; \text{and}\; g^{\prime}(F) = G(g(F)) = G(Z) = Z^{\prime} \]

For any $x \in Z^{\prime}$ choose $y \in Z$ such that $x = G(y)$, then

\[ g^{\prime}(f^{\prime}(x)) = g^{\prime}(f(G^{-1}(G(y)))) = g^{\prime}(f(y)) = G(g(f(y))) = G(y) = x\]
similarly if $x \in F$ then

\[ f^{\prime}(g^{\prime}(x)) = f^{\prime}(G(g(x))) = f(G^{-1}(G(g(x)))) = f(g(x)) = x \]

Thus $g^{\prime}f^{\prime}|Z^{\prime} = \operatorname{id}_{Z^{\prime}}$ and $f^{\prime}g^{\prime}|F = \operatorname{id}_{F}$.

Next note that 
\[ \pi_{1}^{B\cup C}f^{\prime} = \pi_{1}^{B\cup C}fG^{-1} = \pi_{1}^{B\cup C}G^{-1} = \pi_{1}^{B\cup C}\]

\noindent and

\[ \pi_{1}^{B\cup C}g^{\prime} = \pi_{1}^{B\cup C}Gg = \pi_{1}^{B\cup C}g = \pi_{1}^{B\cup C}\]

Also

\begin{multline*}
 \left(\pi_{B\cup C_{0}}^{B\cup C}\times \pi_{B\cup C_{0}}^{B\cup C}\right)f^{\prime} = \left(\pi_{B\cup C_{0}}^{B\cup C}\times \pi_{B\cup C_{0}}^{B\cup C}\right)fG^{-1} = H_{0}\left(\pi_{B\cup C_{0}}^{B\cup C}\times \pi_{B\cup C_{0}}^{B\cup C}\right)G^{-1} =\\ \left(\pi_{B\cup C_{0}}^{B\cup C}\times \pi_{B\cup C_{0}}^{B\cup C}\right)GG^{-1} =
 \pi_{B\cup C_{0}}^{B\cup C}\times \pi_{B\cup C_{0}}^{B\cup C}
\end{multline*}

\noindent Finally 

\begin{multline*}
 \left(\pi_{B\cup C_{0}}^{B\cup C}\times \pi_{B\cup C_{0}}^{B\cup C}\right)g^{\prime} =  \left(\pi_{B\cup C_{0}}^{B\cup C}\times \pi_{B\cup C_{0}}^{B\cup C}\right)Gg = H_{0}\left(\pi_{B\cup C_{0}}^{B\cup C}\times \pi_{B\cup C_{0}}^{B\cup C}\right)g =\\  H_{0}H_{0}^{-1}\left(\pi_{B\cup C_{0}}^{B\cup C}\times \pi_{B\cup C_{0}}^{B\cup C}\right) =  \pi_{B\cup C_{0}}^{B\cup C}\times \pi_{B\cup C_{0}}^{B\cup C}
\end{multline*}

In order to define map $\varphi_{n+1}^{\prime} \colon R^{B\cup C_{n+1}}\times R^{B\cup C_{n+1}} \to R^{B\cup C_{n+1}}\times R^{B\cup C_{n+1}}$ we proceed as follows. First let $L_{n+1} \colon R^{B\cup C_{n+1}}\times R^{B\cup C_{n}} \to R^{B\cup C_{n+1}}\times R^{B\cup C_{n}}$ be a homeomorphism such that $\left(\pi_{B\cup C_{n}}^{B\cup C_{n+1}}\times\operatorname{id}_{R^{B\cup C_{n}}}\right)L_{n+1} = G_{n}\left(\pi_{B\cup C_{n}}^{B\cup C_{n+1}}\times\operatorname{id}_{R^{B\cup C_{n}}}\right)$ and $\pi_{1}L_{n+1} = \pi_{1}$, where $\pi_{1} \colon R^{B\cup C_{n+1}}\times R^{B\cup C_{n}} \to R^{B\cup C_{n+1}}$ is the projection onto the first factor. Let $\varphi_{n+1}^{\prime} = G_{n+1}\varphi_{n+1}iL_{n+1}^{-1}\left(\operatorname{id}_{R^{B\cup C_{n+1}}}\times \pi_{B\cup C_{n}}^{B\cup C_{n+1}}\right)$, where $i \colon R^{B\cup C_{n+1}}\times R^{B\cup C_{n}}\to R^{B\cup C_{n+1}}\times R^{B\cup C_{n+1}}$ is a section of the projection $\operatorname{id}_{R^{B\cup C_{n+1}}}\times \pi_{B\cup C_{n}}^{B\cup C_{n+1}} \colon R^{B\cup C_{n+1}}\times R^{B\cup C_{n+1}} \to R^{B\cup C_{n+1}}\times R^{B\cup C_{n}}$. Note that $\left( \operatorname{id}_{R^{B\cup C_{n+1}}}\times \pi_{B\cup C_{n}}^{B\cup C_{n+1}}\right)G_{n+1} = L_{n+1}\left( \operatorname{id}_{R^{B\cup C_{n+1}}}\times \pi_{B\cup C_{n}}^{B\cup C_{n+1}}\right)$.

Since 

\[ \left(\pi_{B\cup C_{n}}^{B\cup C}\times \pi_{B\cup C_{n}}^{B\cup C}\right)(Z^{\prime}) = \left(\pi_{B\cup C_{n}}^{B\cup C}\times \pi_{B\cup C_{n}}^{B\cup C}\right)(G(Z)) = G_{n}\left(\left(\pi_{B\cup C_{n}}^{B\cup C}\times \pi_{B\cup C_{n}}^{B\cup C}\right)\right)(Z) \]

\noindent  and $\operatorname{Im}(\varphi_{n}^{\prime}) \subseteq G_{n}(\operatorname{Im}(\varphi_{n}))$, it follows that $\left(\pi_{B\cup C_{n}}^{B\cup C}\times \pi_{B\cup C_{n}}^{B\cup C}\right)(Z^{\prime})$ and $\operatorname{Im}(\varphi_{n}^{\prime})$ are functionally separated in $R^{B\cup C_{n}}\times R^{B\cup C_{n}}$, $n \geq 1$.

\begin{multline*}
 \pi_{1}^{B\cup C_{n+1}}\varphi_{n+1}^{\prime} = \pi_{1}^{B\cup C_{n+1}}G_{n+1}\varphi_{n+1}iL_{n+1}^{-1}\left(\operatorname{id}_{R^{B\cup C_{n+1}}}\times \pi_{B\cup C_{n}}^{B\cup C_{n+1}}\right) =\\  \pi_{1}\left( \operatorname{id}_{R^{B\cup C_{n+1}}}\times \pi_{B\cup C_{n}}^{B\cup C_{n+1}}\right)G_{n+1}\varphi_{n+1}iL_{n+1}^{-1}\left(\operatorname{id}_{R^{B\cup C_{n+1}}}\times \pi_{B\cup C_{n}}^{B\cup C_{n+1}}\right) =\\ \pi_{1}L_{n+1}\left( \operatorname{id}_{R^{B\cup C_{n+1}}}\times \pi_{B\cup C_{n}}^{B\cup C_{n+1}}\right)\varphi_{n+1}iL_{n+1}^{-1}\left(\operatorname{id}_{R^{B\cup C_{n+1}}}\times \pi_{B\cup C_{n}}^{B\cup C_{n+1}}\right) = \\ \pi_{1}L_{n+1}\left( \operatorname{id}_{R^{B\cup C_{n+1}}}\times \pi_{B\cup C_{n}}^{B\cup C_{n+1}}\right)iL_{n+1}^{-1}\left(\operatorname{id}_{R^{B\cup C_{n+1}}}\times \pi_{B\cup C_{n}}^{B\cup C_{n+1}}\right) =\\ \pi_{1}L_{n+1}L_{n+1}^{-1}\left(\operatorname{id}_{R^{B\cup C_{n+1}}}\times \pi_{B\cup C_{n}}^{B\cup C_{n+1}}\right) = \pi_{1}\left(\operatorname{id}_{R^{B\cup C_{n+1}}}\times \pi_{B\cup C_{n}}^{B\cup C_{n+1}}\right) = \pi_{1}^{B\cup C_{n+1}}
\end{multline*}

\noindent and

\begin{multline*}
 \left(\pi_{B\cup C_{n}}^{B\cup C_{n+1}}\times \pi_{B\cup C_{n}}^{B\cup C_{n+1}}\right)\varphi_{n+1}^{\prime} =\\  \left(\pi_{B\cup C_{n}}^{B\cup C_{n+1}}\times \operatorname{id}_{R^{B\cup C_{n}}}\right) \left( \operatorname{id}_{R^{B\cup C_{n+1}}}\times \pi_{B\cup C_{n}}^{B\cup C_{n+1}}\right)G_{n+1}\varphi_{n+1}iL_{n+1}^{-1}\left(\operatorname{id}_{R^{B\cup C_{n+1}}}\times \pi_{B\cup C_{n}}^{B\cup C_{n+1}}\right) =\\
\left(\pi_{B\cup C_{n}}^{B\cup C_{n+1}}\times \operatorname{id}_{R^{B\cup C_{n}}}\right)L_{n+1}\left( \operatorname{id}_{R^{B\cup C_{n+1}}}\times \pi_{B\cup C_{n}}^{B\cup C_{n+1}}\right)\varphi_{n+1}iL_{n+1}^{-1}\left(\operatorname{id}_{R^{B\cup C_{n+1}}}\times \pi_{B\cup C_{n}}^{B\cup C_{n+1}}\right) = \\
\left(\pi_{B\cup C_{n}}^{B\cup C_{n+1}}\times \operatorname{id}_{R^{B\cup C_{n}}}\right)L_{n+1}\left( \operatorname{id}_{R^{B\cup C_{n+1}}}\times \pi_{B\cup C_{n}}^{B\cup C_{n+1}}\right)iL_{n+1}^{-1}\left(\operatorname{id}_{R^{B\cup C_{n+1}}}\times \pi_{B\cup C_{n}}^{B\cup C_{n+1}}\right) =\\
 \left(\pi_{B\cup C_{n}}^{B\cup C_{n+1}}\times \operatorname{id}_{R^{B\cup C_{n}}}\right)L_{n+1}L_{n+1}^{-1}\left(\operatorname{id}_{R^{B\cup C_{n+1}}}\times \pi_{B\cup C_{n}}^{B\cup C_{n+1}}\right) =\\ \left(\pi_{B\cup C_{n}}^{B\cup C_{n+1}}\times \operatorname{id}_{R^{B\cup C_{n}}}\right)\left(\operatorname{id}_{R^{B\cup C_{n+1}}}\times \pi_{B\cup C_{n}}^{B\cup C_{n+1}}\right) = \pi_{B\cup C_{n}}^{B\cup C_{n+1}}\times \pi_{B\cup C_{n}}^{B\cup C_{n+1}}
\end{multline*}

By Lemma \ref{L:limitfiberwise}, there exists a homeomorphism $S \colon R^{B\cup C}\times R^{B\cup C} \to R^{B\cup C}\times R^{B\cup C}$ such that $\pi_{1}^{B\cup C}S = \pi_{1}^{B\cup C}$, $\left(\pi_{B\cup C_{0}}^{B\cup C}\times \pi_{B\cup C_{0}}^{B\cup C}\right) S = \pi_{B\cup C_{0}}^{B\cup C}\times \pi_{B\cup C_{0}}^{B\cup C}$ and $S|Z^{\prime} = f^{\prime}|Z^{\prime}$. It only remains to observe that the homeomorphism $H = SG$ has all the required properties.
\end{proof}

\begin{thm}\label{T:ztauunknottingsigmaf}
Let $\tau \geq \omega$ and $Z$ and $F$ be $C$-embedded fibered $Z_{\tau}$-sets of $R^{\tau} \times R^{\tau}$ with respect to the projection $\pi_{1} \colon R^{\tau} \times R^{\tau} \to R^{\tau}$. Let also $h \colon Z \to F$ be a homeomorphism such that $\pi_{1}h = \pi_{1}|Z$. Then there exists a homeomorphism $H \colon R^{\tau}\times R^{\tau} \to R^{\tau}\times R^{\tau}$ such that $H|Z = h$ and $\pi_{1}H = \pi_{1}$.
\end{thm}
\begin{proof}
Let $|A| = \tau$. As noted above, for $\tau = \omega$ statement is known. Below we assume that $\tau > \omega$. 

Since the projection $\pi_{1}^{A} \colon R^{A} \times R^{A}\to R^{A}$ is soft, $\pi_{1}^{A}h = \pi_{1}^{A}|Z$ and $Z$ is $C$-embedded in $R^{A}$, it follows from definition of softness \cite[Proposition 6.1.16]{chibook} that there exists a map $f \colon R^{A}\times R^{A} \to R^{A}\times R^{A}$ such that $\pi_{1}^{A}f = \pi_{1}^{A}$ and $f|Z = h$. Similarly there is a map $g \colon R^{A}\times R^{A} \to R^{A}\times R^{A}$ such that $\pi_{1}^{A}g = \pi_{1}^{A}$ and $g|F = h^{-1}$. Note that since $\pi_{1}^{A}f = \pi_{1}^{A} = \pi_{1}^{A}g$, for any $C \in {\mathcal M}_{f,g}$ in addition to properties listed in Lemma \ref{L:2}, we also have $\pi_{1}^{C}f = \pi_{1}^{C} = \pi_{1}^{C}g$.

Let $A= \{ a_{\alpha} \colon \alpha < \tau\}$ be a well-ordering of $A$. 

Our first goal is to construct a countable subset $A_{0} \in {\mathcal M}_{f,g}$ such that $a_{0} \in A_{0}$ and a homeomorphism $H_{0} \colon R^{A_{0}}\times R^{A_{0}} \to R^{A_{0}}\times R^{A_{0}}$ such that $H_{0}|\operatorname{cl}_{R^{A_{0}}\times R^{A_{0}}}\left(\pi_{A_{0}}^{A}\times \pi_{A_{0}}^{A}\right)(Z) = f_{A_{0}}|\operatorname{cl}_{R^{A_{0}}\times R^{A_{0}}}\left(\pi_{A_{0}}^{A}\times \pi_{A_{0}}^{A}\right)(Z)$.

 Start by choosing $C_{0} \in {\mathcal M}_{f,g}$ such that $a_{0} \in C_{0}$. Since $Z$ and $F$ are a fibered $Z_{\tau}$-sets, there exist, by Lemma \ref{L:description}, maps $\varphi, \psi \colon R^{A}\times R^{A} \to R^{A}\times R^{A}$ such that 

\begin{enumerate}
  \item
$\left(\pi_{C_{0}}^{A}\times \pi_{C_{0}}^{A}\right)\varphi = \pi_{C_{0}}^{A}\times \pi_{C_{0}}^{A}$; 
  \item 
$\left(\pi_{C_{0}}^{A}\times \pi_{C_{0}}^{A}\right)\psi = \pi_{C_{0}}^{A}\times \pi_{C_{0}}^{A}$; 
  \item 
$\pi_{1}^{A}\varphi = \pi_{1}^{A}$;
  \item 
$\pi_{1}^{A}\psi = \pi_{1}^{A}$;
  \item 
 $Z$ and $\operatorname{Im}(\varphi)$ are functionally separated in $R^{A}\times R^{A}$;
  \item 
 $F$ and $\operatorname{Im}(\psi)$ are functionally separated in $R^{A}\times R^{A}$.
\end{enumerate}  

Using Lemmas \ref{L:1} and \ref{L:2}, we can find a set $C_{1} \in {\mathcal M}_{f,g}\cap {\mathcal M}_{\varphi}\cap {\mathcal M}_{\psi}$ and maps $\varphi_{1}, \psi_{1} \colon R^{C_{1}}\times R^{C_{1}} \to R^{C_{1}}\times R^{C_{1}}$ such that $\left(\pi_{C_{1}}^{A}\times \pi_{C_{1}}^{A}\right)\varphi = \varphi_{1}\left(\pi_{C_{1}}^{A}\times \pi_{C_{1}}^{A}\right)$ and $\left(\pi_{C_{1}}^{A}\times \pi_{C_{1}}^{A}\right)\psi = \psi_{1}\left(\pi_{C_{1}}^{A}\times \pi_{C_{1}}^{A}\right)$. It is easy to see that
\begin{itemize}
  \item[(7)]
$\pi_{1}^{C_{1}}\varphi_{1} = \pi_{1}^{C_{1}}$;
  \item[(8)]
$\left(\pi_{C_{0}}^{C_{1}}\times \pi_{C_{0}}^{C_{1}}\right)\varphi_{1} = \left(\pi_{C_{0}}^{C_{1}}\times \pi_{C_{0}}^{C_{1}}\right)$;
  \item[(9)] 
$\pi_{1}^{C_{1}}\psi_{1} = \pi_{1}^{C_{1}}$;
  \item[(10)]
$\left(\pi_{C_{0}}^{C_{1}}\times \pi_{C_{0}}^{C_{1}}\right)\psi_{1} = \left(\pi_{C_{0}}^{C_{1}}\times \pi_{C_{0}}^{C_{1}}\right)$;
  \item[(11)]
$\left(\pi_{C_{1}}^{A}\times\pi_{C_{1}}^{A}\right)(Z)$ and $\operatorname{Im}(\varphi_{1})$ have disjoint closures in $R^{C_{1}}\times R^{C_{1}}$; 
  \item[(12)] 
$\left(\pi_{C_{1}}^{A}\times\pi_{C_{1}}^{A}\right)(F)$ and $\operatorname{Im}(\psi_{1})$ have disjoint closures in $R^{C_{1}}\times R^{C_{1}}$ 
\end{itemize}

Continuing in this manner we can construct increasing sequence $\{ C_{n} \colon n \in \omega\}$ of elements of ${\mathcal M}_{f,g}$ and maps $\varphi_{n}, \psi_{n} \colon R^{C_{n}}\times R^{C_{n}} \to R^{C_{n}}\times R^{C_{n}}$ satisfying the following properties:

\begin{itemize}
  \item[(i)]
$\pi_{1}^{C_{n+1}}\varphi_{n+1} = \pi_{1}^{C_{n+1}}$;
  \item[(ii)]
$\left(\pi_{C_{n}}^{C_{n+1}}\times \pi_{C_{n}}^{C_{n+1}}\right)\varphi_{n+1} = \pi_{C_{n}}^{C_{n+1}}\times \pi_{C_{n}}^{C_{n+1}}$;
  \item[(iii)] 
$\pi_{1}^{C_{n+1}}\psi_{n+1} = \pi_{1}^{C_{n+1}}$;
  \item[(iv)]
$\left(\pi_{C_{n}}^{C_{n+1}}\times \pi_{C_{n}}^{C_{n+1}}\right)\psi_{n+1} = \pi_{C_{n}}^{C_{n+1}}\times \pi_{C_{n}}^{C_{n+1}}$;
  \item[(v)]
$\left(\pi_{C_{n+1}}^{A}\times \pi_{C_{n+1}}^{A}\right)(Z)$ and $\operatorname{Im}(\varphi_{n+1})$ have disjoint closures in $R^{C_{n+1}}\times R^{C_{n+1}}$; 
  \item[(vi)] 
$\left(\pi_{C_{n+1}}^{A}\times \pi_{C_{n+1}}^{A}\right)(F)$ and $\operatorname{Im}(\psi_{n+1})$ have disjoint closures in $R^{C_{n+1}}\times R^{C_{n+1}}$; 
\end{itemize}

Let $A_{0} = \cup\{ C_{n} \colon n \in \omega\}$. Note that $A_{0} \in {\mathcal M}_{f,g}$. By Lemma \ref{L:limitfiberwise}, the homeomorphism 

\[ f_{A_{0}}|... \colon \operatorname{cl}_{{\mathbb R}^{A_{0}}\times R^{A_{0}}}\left(\pi^{A}_{A_{0}}\times \pi_{A_{0}}^{A}\right)(Z) \to \operatorname{cl}_{{\mathbb R}^{A_{0}}\times R^{A_{0}}}\left(\pi^{A}_{A_{0}}\times \pi_{A_{0}}^{A}\right)(F) \] 

\noindent can be extended to a homeomorphism $H_{0} \colon R^{A_{0}}\times R^{A_{0}} \to R^{A_{0}}\times R^{A_{0}}$ so that $\pi_{1}^{A_{0}}H_{0} = \pi_{1}^{A_{0}}$.

Suppose that for each $\beta < \alpha < \tau$ we have already constructed subsets $A_{\beta} \subseteq A$ and autohomeomorphisms $H_{\beta} \colon R^{A_{\beta}}\times R^{A_{\beta}} \to R^{A_{\beta}}\times R^{A_{\beta}}$, $\beta < \tau$, satisfying the following properties:
\begin{itemize}
  \item[(1)$_{\beta}$]
 $A_{\beta} \cup\{ a_{\beta}\} \subseteq A_{\beta +1}$ and $|A_{\beta}| < \tau$ for each $\beta < \alpha$;
  \item[(2)$_{\beta}$] 
 $A_{\beta} = \cup\{ A_{\gamma} \colon \gamma < \beta\}$ whenever $\beta <\alpha$ is a limit ordinal;
  \item[(4)$_{\beta}$] 
 $\left(\pi_{A_{\beta}}^{A_{\beta +1}}\times \pi_{A_{\beta}}^{A_{\beta +1}}\right)H_{\beta+1} = H_{\beta}\left(\pi_{A_{\beta}}^{A_{\beta +1}}\times \pi_{A_{\beta}}^{A_{\beta +1}}\right)$ for each $\beta < \alpha$;
  \item[(5)$_{\beta}$]
 $H_{\beta} = \lim\{ H_{\gamma} \colon \gamma < \beta\}$ whenever $\beta <\alpha$ is a limit ordinal;
  \item[(6)$_{\beta}$]
 $\left(\pi_{A_{\beta}}^{A_{\beta +1}} \times \pi_{A_{\beta}}^{A_{\beta +1}}\right)H_{\beta +1} = H_{\beta}\left(\pi_{A_{\beta}}^{A_{\beta +1}} \times \pi_{A_{\beta}}^{A_{\beta +1}}\right)$ ; 
\item[(7)$_{\beta}$]
$\pi_{1}^{A_{\beta}}H_{\beta} = \pi_{1}^{A_{\beta}}$;
  \item[(8)$_{\beta}$]
  $H_{\beta}|\operatorname{cl}_{{\mathbb R}^{A_{\beta}}\times R^{A_{\beta}}}\left(\pi_{A_{\beta}}^{A}\times \pi_{A_{\beta}}^{A}\right)(Z) = f_{A_{\beta}}|\operatorname{cl}_{{\mathbb R}^{A_{\beta}}\times R^{A_{\beta}}}\left(\pi_{A_{\beta}}^{A}\times \pi_{A_{\beta}}^{A}\right)(Z)$ for each $\beta < \alpha$.  
\end{itemize}

In order to complete inductive step we need to construct $A_{\alpha}$ and $H_{\alpha}$.

Case $\alpha = \beta +1$. Let $C_{\beta, 0} = \emptyset$. Assuming that $C_{\beta,n} \in {\mathcal M}_{f,g}$ has already been constructed, we proceed as follows. Since $|A_{\beta} \cup C_{\beta,n}|< \tau$ and since $Z$ and $F$ are fibered $Z_{\tau}$-sets, there exist maps $\varphi, \psi \colon R^{A} \times R^{A}\times R^{A}$ such that $\pi_{1}^{A}\varphi = \pi_{1}^{A} = \pi_{1}^{A}\psi$, $\left( \pi_{A_{\beta}\cup C_{\beta,n}}^{A} \times \pi_{A_{\beta}\cup C_{\beta,n}}^{A}\right)\varphi = \pi_{A_{\beta}\cup C_{\beta,n}}^{A} \times \pi_{A_{\beta}\cup C_{\beta,n}}^{A} = \left( \pi_{A_{\beta}\cup C_{\beta}}^{A} \times \pi_{A_{\beta}\cup C_{\beta}}^{A}\right)\psi$, $Z$ and $\operatorname{Im}(\varphi)$, as well as $F$ and $\operatorname{Im}(\psi)$ are functionally separated in $R^{A}\times R^{A}$. Let $C_{\beta, n+1}$ be any element of  ${\mathcal M}_{f,g,}\cap {\mathcal M}_{\varphi}\cap {\mathcal M}_{\psi}$, containing $a_{\beta}$.  Let $A_{\beta +1} = A_{\beta} \cup \left(\cup\{ C_{\beta,n}\colon n \in \omega\}\right)$. Lemma \ref{L:limitfiberwise1} guarantees that there exists a homeomorphism $H_{\beta +1} \colon R^{A_{\beta +1}}\times R^{A_{\beta +1}} \to R^{A_{\beta +1}}\times R^{A_{\beta +1}}$ satisfying the needed conditions for the ordinal $\beta +1$.

Case of a limit ordinal $\alpha$. Let $A_{\alpha} = \cup\{A_{\beta} \colon \beta < \alpha\}$ and $H_{\alpha} = \lim\{ H_{\beta} \colon \beta < \alpha\}$.

This completes inductive construction of $A_{\alpha}$ and $H_{\alpha}$, $\alpha < \tau$. The required homeomorphism $H \colon R^{A}\times R^{A} \to R^{A}\times R^{A}$ can now be defined as $H = \lim\{ H_{\alpha} \colon \alpha < \tau\}$. 
\end{proof}

\begin{cor}\label{C:1}
Let $\tau > \omega$ and $h \colon Z \to F$ be a homeomorphism between $C$-embedded $Z_{\tau}$-sets in $R^{\tau}$. Then there is an autohomeomorphism $H$ of $R^{\tau}$ such that $H|Z = h$.
\end{cor}
\begin{proof}
Let $a \in R^{\tau}$. Note that $Z_{a} = \{ a\} \times Z$ and $F_{a} = \{ a \}\times F$ are $C$-embedded fibered $Z_{\tau}$-sets in $R^{\tau}\times R^{\tau}$ with respect to the projection $\pi_{1} \colon R^{\tau} \times R^{\tau} \to R^{\tau}$. The homeomorphism $h_{a} \colon Z_{a} \to F_{a}$, defined by letting $h_{a}(a,x) = (a,h(x))$, $x \in Z$, acts fiberwise with respect to the projection $\pi_{1}$. By Theorem \ref{T:ztauunknottingsigmaf}, there is a homeomorphism $H_{a} \colon R^{\tau}\times R^{\tau} \to R^{\tau}\times R^{\tau}$ such that $H_{a}|Z_{a}= h_{a}$ and $\pi_{1}H_{a} = \pi_{1}$. Then the required extension of $h$ can be defined by letting $H(x) = H_{a}(a,x)$, $x \in R^{\tau}$.
\end{proof}

We want to note that Theorem \ref{T:ztauunknottingsigmaf} remains true for the projection $\pi_{1} \colon X \times R^{\tau} \to X$ where $X$ is an $R^{\tau}$-manifold. Similarly homeomorphisms between $Z_{\tau}$-sets of an $R^{\tau}$-manifold admit extensions to the ambient manifold.

The following statement for compact $Z$ and $F$ was proved in \cite{vds}.

\begin{cor}\label{C:iff}
Let $\tau > \omega$ and $Z$ and $F$ be closed $C$-embedded $Z_{\tau}$-sets in $R^{\tau}$. Then the following conditions are equivalent:
\begin{itemize}
  \item[(i)]
 $Z$ and $F$ are homeomorphic;
  \item[(ii)] 
 $R^{\tau}\setminus Z$ and $R^{\tau}\setminus F$ are homeomorphic.  
\end{itemize}
\end{cor}
\begin{proof}
(i) $\Longrightarrow$ (ii) follows from Corollary \ref{C:1}. In order to prove (ii) $\Longrightarrow$ (i) it suffices to note that neither $Z$ nor $F$ contains $G_{\delta}$-subsets of $R^{\tau}$ and consequently $R^{\tau}$ is the Hewitt realcompactification of both $R^{\tau}\setminus Z$ and $R^{\tau}\setminus F$. Thus any homeomorphism $h \colon R^{\tau}\setminus F \to R^{\tau}\setminus F$ can be uniquely extended to a homeomorphism $\widetilde{h} \colon R^{\tau} \to R^{\tau}$. Then $\widetilde{h}(Z) = F$ and $\widetilde{h}|Z \colon Z \to F$ is the required homeomorphism.
\end{proof}

%%%%%%%%%%%%%%%%%%%%%%%%%%%%%%%%%%%%%%%%%%%%%%%%%%%

\subsection{Compact case}\label{SS:compact}
Precisely same proof (even its simplified version) with just one adjustment -- replacing Theorem \ref{T:fz-setunknottingR} by Theorem \ref{T:fz-setunknottingI} in the proof of Lemma \ref{L:3} -- is needed for obtaining $Z_{\tau}$-set unknotting results in the compact case.

\begin{thm}\label{T:ztauunknottingsigmafI}
Let $\tau \geq \omega$ and $Z$ and $F$ be fibered $Z_{\tau}$-sets of $I^{\tau} \times I^{\tau}$ with respect to the projection $\pi_{1} \colon I^{\tau} \times I^{\tau} \to I^{\tau}$. Let also $h \colon Z \to F$ be a homeomorphism such that $\pi_{1}h = \pi_{1}|Z$. Then there exists a homeomorphism $H \colon I^{\tau}\times I^{\tau} \to I^{\tau}\times I^{\tau}$ such that $H|Z = h$ and $\pi_{1}H = \pi_{1}$.
\end{thm}

A version of the following statement appears in \cite{M}
\begin{cor}\label{C:1I}
Let $\tau > \omega$ and $h \colon Z \to F$ be a homeomorphism between $Z_{\tau}$-sets in $I^{\tau}$. Then there is an autohomeomorphism $H$ of $I^{\tau}$ such that $H|Z = h$.
\end{cor}

%%%%%%%%%%%%%%%%%%%%%%%%%%%%%%%%%%%%%%%%%%%%%%%%%%%%%%%%

\section{Recognizing $Z_{\tau}$-sets}\label{S:recognizing}
Problem of recognizing $Z_{\tau}$-sets is an important one if one wishes to use unknotting theorems. It is relatively straightforward to detect whether $\operatorname{id}_{R^{\tau}}$ can be approximated by self maps images of which miss a given set $Z$. It is somewhat harder to find approximating maps whose images are functionally separated from $Z$. Of course, in case of the Tychonov cube the latter problem is redundant. We start by addressing the former.

\begin{lem}\label{L:disjointsection}
Let $|A| = \tau > \omega$ and $Z$ be a closed subset of $R^{A}$ containing no closed $G_{\kappa}$-subsets of $R^{A}$ for any $\kappa < \tau$. Then for each $B \subseteq A$ with $|B| < \tau$ there is a section $i_{B}^{A} \colon R^{B} \to R^{A}$ of the projection $\pi_{B}^{A} \colon R^{A} \to R^{B}$ such that $Z \cap i_{B}^{A}(R^{B}) = \emptyset$. 
\end{lem}
\begin{proof}
Let $\kappa = \max\{ |B|, \omega\}$, $B_{0} = B$ and $x_{0} \in R^{B _{0}}$. By assumption, $\left(\pi_{B_{0}}^{A}\right)^{-1}(x_{0}) \setminus Z \neq \emptyset$. Choose a subset $B_{1} \subseteq A$ such that $B_{0} \subseteq B_{1}$, $|B_{1}\setminus B_{0}| \leq \omega$ and $\left(\pi_{B_{0}}^{B_{1}}\right)^{-1}(x_{0}) \setminus \operatorname{cl}_{R^{B_{1}}}(\pi_{B_{1}}^{A}(Z)) \neq \emptyset$. Let $x_{1} \in \left(\pi_{B_{0}}^{B_{1}}\right)^{-1}(x_{0}) \setminus \operatorname{cl}_{R^{B_{1}}}(\pi_{B_{1}}^{A}(Z))$. Choose a section $i_{0}^{1} \colon R^{B_{0}} \to R^{B_{1}}$ of the projection $\pi_{B_{0}}^{B_{1}} \colon R^{B_{1}} \to R^{B_{0}}$ such that $i_{0}^{1}(x_{0}) = x_{1}$. Let 

\[ V_{1} = \{ x \in R^{B_{0}}\colon i_{0}^{1}(x) \notin \operatorname{cl}_{R^{B_{1}}}(\pi_{B_{1}}^{A}(Z))\}.\]

\noindent Note that $x_{0} \in V_{1}$ and consequently $V_{1}$ is a non-empty open subset of $R^{B_{0}}$.

Let $\gamma < \kappa^{+}$. Suppose that for each $\lambda$, $1 \leq \lambda < \gamma$, we have already constructed a subset $B_{\lambda} \subseteq A$, an open subset $V_{\lambda} \subseteq R^{B_{0}}$ and a section $i_{0}^{\lambda} \colon R^{B_{0}} \to R^{B_{\lambda}}$ of the projection $\pi_{B_{0}}^{B_{\lambda}} \colon R^{B_{\lambda}} \to R^{B_{0}}$, satisfying the following conditions:
\begin{itemize}
\item[(i)]
$B_{\lambda} < B_{\mu}$, whenever $\lambda < \mu < \gamma$,
\item[(ii)]
$B_{\mu} = \cup\{ B_{\lambda} \colon \lambda < \mu\}$, whenever $\mu < \gamma$ is a limit ordinal,
\item[(iii)]
$V_{\lambda} \subseteq V_{\mu}$, whenever $\lambda < \mu < \gamma$,
\item[(iv)]
$V_{\mu} = \cup\{ V_{\lambda} \colon \lambda < \mu\}$, whenever $\mu < \gamma$ is a limit ordinal,
\item[(v)]
$i_{0}^{\mu} = \lim\{ i_{0}^{\lambda} \colon \lambda < \mu\}$, whenever $\mu < \gamma$ is a limit ordinal,
\item[(vi)]
$i_{0}^{\lambda} = \pi_{B_{\lambda}}^{B_{\mu}}i_{0}^{\mu}$, whenever $\lambda < \mu < \gamma$,
\item[(vii)]
$V_{\lambda} = \{ x \in R^{B_{0}}\colon i_{0}^{\lambda}(x) \notin \operatorname{cl}_{R^{B_{\lambda}}}(\pi_{B_{\lambda}}(Z))\}$
\end{itemize}
We shall construct a set $B_{\gamma}$, an open subset $V_{\gamma} \subseteq R^{B_{0}}$ and a section $i_{0}^{\gamma} \colon R^{B_{0}} \to R^{B_{\gamma}}$ of the projection $\pi^{B_{\gamma}}_{B_{0}} \colon R^{B_{\gamma}} \to R^{B_{0}}$.

Suppose that $\gamma$ is a limit ordinal. Then we let $ B_{\gamma} = \cup\{ B_{\mu} \colon \mu < \gamma\}$ and 
\[ i_{0}^{\gamma} = \lim\{ i_{0}^{\mu} \colon \mu < \gamma\} \colon R^{B_{0}}\to R^{B_{\gamma}}\]

\noindent is well defined and satisfies corresponding conditions (v) and (vi). Let 
\[ V_{\gamma} = \{ x \in X_{\alpha_{0}}\colon i_{0}^{\gamma}(x) \notin p_{\alpha_{\gamma}}(F)\} .\]

\noindent Note that $V_{\gamma} = \cup\{ V_{\alpha_{\mu}} \colon \mu < \gamma\}$.

Next consider the case $\gamma = \mu +1$. Suppose that $V_{\mu} \neq R{B_{0}}$ and let
\[ x_{\mu} = i_{0}^{\mu}(z) \in i_{0}^{\mu}(R^{B_{0}}) \subseteq R^{B_{\mu}},\]

\noindent where $z \in R^{B_{0}}\setminus V_{\mu}$. By assumption, $\left(\pi_{B_{\mu}}^{A}\right)^{-1}(x_{\mu}) \setminus Z \neq \emptyset$ (note that $w(X_{\alpha_{\mu}}) \leq \kappa$). Choose a set $B_{\gamma} \subseteq A$ so that $B_{\gamma} \supseteq B_{\mu}$, $|B_{\gamma}\setminus B_{\mu}| \leq \omega$ and $\left(\pi_{B_{\mu}}^{B_{\gamma}}\right)^{-1}(x_{\mu}) \setminus \operatorname{cl}_{R^{B_{\gamma}}}(\pi_{B_{\gamma}}(Z)) \neq \emptyset$.

Take any section $i_{\mu}^{\gamma} \colon R^{B_{\mu}} \to R^{B_{\gamma}}$ of the projection $\pi_{B_{\mu}}^{B_{\gamma}} \colon R^{B_{\gamma}} \to R^{B_{\mu}}$ such that $i_{\mu}^{\gamma}(x_{\mu}) = z^{\prime}$, where $z^{\prime} \in \left(\pi_{B_{\mu}}^{B_{\gamma}}\right)^{-1}(x_{\mu}) \setminus \operatorname{cl}_{R^{B_{\gamma}}}(\pi_{B_{\gamma}}(Z))$. Let $i_{0}^{\gamma} = i_{\mu}^{\gamma}i_{0}^{\mu}$ and $V_{\gamma} = \{ x \in R^{B_{0}}\colon i_{0}^{\gamma}(x) \notin \operatorname{cl}_{R^{B_{\gamma}}}(\pi_{B_{\gamma}}(Z))\}$. Note that $V_{\mu} \subseteq V_{\gamma}$ and $z \in V_{\gamma}\setminus V_{\mu}$. This completes construction of the needed objects in the case $\gamma = \mu +1$.

Thus the construction can be carried out for each $\lambda < \kappa^{+}$ and we obtain an increasing collection $\{ V_{\lambda} \colon \lambda < \kappa^{+}\}$ of length $\kappa^{+}$ of open subsets of $R^{B_{0}}$. Since $|B_{0}| \leq \kappa$, this collection must stabilize, which means that there in an index $\lambda_{0} < \kappa^{+}$ such that $V_{\lambda} = V_{\lambda_{0}}$ for any $\lambda \geq \lambda_{0}$.
By construction, this is only possible if $V_{\lambda_{0}} = R^{B_{0}}$. Let $i_{B}^{A} = i_{B_{\lambda_{0}}}^{A}i_{B_{0}}^{B_{\lambda_{0}}}$, where $i_{B_{\lambda_{0}}}^{A} \colon R^{B_{\lambda_{0}}} \to R^{A}$ is any section of the projection $\pi_{B_{\lambda_{0}}}^{A} \colon R^{A} \to R^{B_{\lambda_{0}}}$. Clearly $i_{B}^{A}(R^{B}) \cap Z = \emptyset$.
\end{proof}

It turns out that recognizing $Z_{\omega_{1}}$-sets in $R^{\omega_{1}}$ is very simple.

\begin{pro}\label{P:recognize}
The following conditions are equivalent for a closed subset $Z$ of $R^{\omega_{1}}$:
\begin{itemize}
  \item[(i)]
 $Z$ is a $Z_{\omega_{1}}$-set;
  \item[(ii)] 
$Z$ does not contain $G_{\delta}$-subsets of $R^{\omega_{1}}$.  
\end{itemize}
\end{pro}
\begin{proof}
(i) $\Longrightarrow$ (ii). Assuming that $Z$ does contain a $G_{\delta}$-subset of $R^{\omega_{1}}$, we can find a countable subset $B \subseteq \omega_{1}$ and a point $x_{0} \in R^{B}$ such that $\pi_{B}^{-1}(x_{0}) \subseteq Z$. On the other hand, by (i), there is a map $f \colon R^{\omega_{1}} \to R^{\omega_{1}}$ such that $\pi_{B}f = \pi_{B}$ and $Z \cap f(R^{\omega_{1}}) = \emptyset$. Take a point $y_{0} \in R^{\omega_{1}}$ such that $\pi_{B}(y_{0}) = x_{0}$. Clearly $f(y_{0}) \in \pi_{B}^{-1}(x_{0})$ and consequently $f(y_{0}) \in Z$. But this is impossible since $Z \cap f(R^{\omega_{1}}) = \emptyset$.

(ii) $\Longrightarrow$ (i). For an arbitrary countable subset $B \subseteq \omega_{1}$, we can find, using Lemma \ref{L:disjointsection}, a section $i_{B} \colon R^{B} \to R^{\omega_{1}}$ of the projection $\pi_{B} \colon R^{\omega_{1}} \to R^{B}$ such that $Z \cap i_{B}(R^{B}) = \emptyset$. Since $i_{B}(R^{B})$ is Lindel\"{o}f there is a functionally open subset $U$ of $R^{\omega_{1}}$ such that $i_{B}(R^{B}) \subseteq U$ and $U \cap Z = \emptyset$. Since $i_{B}(R^{B})$ is $C$-embedded in $R^{\omega_{1}}$ it follows that $i_{B}(R^{B})$ is functionally separated from $R^{\omega_{1}}\setminus U$. Then the map $f = i_{B}\pi_{B}$ witnesses the fact that $Z$ is a $Z_{\tau}$-set.
\end{proof} 

In the compact case we have the similar equivalence for any weight.
\begin{pro}\label{P:recognizeI}
Let $\tau > \omega$. The following conditions are equivalent for a closed subset $Z$ of $I^{\tau}$:
\begin{itemize}
  \item[(i)]
 $Z$ is a $Z_{\tau}$-set;
  \item[(ii)] 
$Z$ does not contain $G_{\kappa}$-subsets of $R^{\tau}$ for any $\kappa < \tau$.  
\end{itemize}
\end{pro}

\begin{pro}\label{P:lindelof}
Let $\tau > \omega$. Any closed $C$-embedded Lindel\"{o}f subset of $R^{\tau}$ is a $Z_{\tau}$-set. 
\end{pro}
\begin{proof}
Let $|A| = \tau$ and $Z$ be a closed $C$-embedded Lindel\"{o}f subset of $R^{\tau}$. Since any $G_{\kappa}$-subset of $R^{A}$ with $\kappa < \tau$ contains a closed copy of $R^{A}$ and since $R^{A}$ is not Lindel\"{o}f, it follows that $Z$ does not contain any $G_{\kappa}$-subset of $R^{A}$. For any $B \subseteq A$ with $|B| < \tau$ we can, according to Lemma \ref{L:disjointsection}, find a section $i_{B}^{A} \colon R^{B} \to R^{A}$ of the projection $\pi_{B}^{A} \colon R^{A} \to R^{B}$ such that $Z \cap i_{B}^{A}(R^{B}) = \emptyset$. Since $i_{B}^{A}(R^{B})$ is closed in $R^{A}$ and since $Z$ is Lindel\'{o}f, there is a functionally open subset $U$ of $R^{A}$ such that $Z \subseteq U$ and $U \cap i_{B}^{A}(R^{B})=\emptyset$. Since $Z$ is $C$-embedded it follows that $Z$ is functionally separated from $R^{A}\setminus U$. Consequently $Z$ and $i_{B}^{A}(R^{B})$ are also functionally separated in $R^{A}$. Thus the map $f \colon R^{A} \to R^{A}$, defined by letting $f = i_{B}^{A}\pi_{B}^{A}$, witnesses the fact that $Z$ is a $Z_{\tau}$-set.
\end{proof}

We conclude by positively answering a question from \cite{vds}. Proof of the following statement follows from Corollary \ref{C:1} and Proposition \ref{P:lindelof}.

\begin{cor}\label{C:2}
Let $\tau > \omega$. Any homeomorphism between closed $C$-embedded and Lindel\"{o}f subsets of $R^{\tau}$ can be extended to an autohomeomorphism of $R^{\tau}$.
\end{cor}

%%%%%%%%%%%%%%%%%%%%%%%%%%%%%%%%%%%%%%%%%%%%%%
%%%%%%%%%%%%%%%%%%%%%%%%%%%%%%%%%%%%%%%%%%%%

\end{document}